\documentclass{article}
\usepackage[utf8]{inputenc}
\usepackage{ dsfont }
\usepackage{indentfirst}
\usepackage{ amssymb }
\usepackage{mathtools}
\usepackage{amsmath}
\usepackage{amsthm}
\usepackage{graphicx}
\graphicspath{ {./images/} }

\theoremstyle{definition}
\newtheorem*{definition}{Definition}

\title{Lovasz' Conjecture and Other Applications of Topological Methods in Discrete Mathematics}
\author{Jingsi Hou, Guangyan Huang, Sammy Suliman, Haoran Yan}
\date{September 2022}

\begin{document}

\maketitle

\section{Background}
In 20th century mathematics, the field of topology, which concerns the properties of geometric objects under continuous transformation, has proved surprisingly useful in application to the study of discrete mathematics, such as combinatorics, graph theory, and theoretical computer science. In this paper, we seek to provide an introduction to the relevant topological concepts to non-specialists, as well as a selection of some existing applications to theorems in discrete mathematics.

\section{Introductory Topology}
Topology is the study of continuously deformable objects. To rigorously define this, we begin by introducing the concept of a \textbf{homeomorphism}. Two objects are homeomorphic if there exist a continuous bijection between them with a continuous inverse. For example, a square and a circle are homeomorphic. While an explicit map may not immediately clear (see [1]), its existence is made intuitively evident by inflating the edges of the square into a circle. On the other hand, a homeomorphism does not exist between the torus, the geometric shape formed by revolving a circle in $\mathds{R}^3$ about an axis (visually, a hollow doughnut), and a sphere in $\mathds{R}^3$. If X and Y are two homeomorphic topological spaces, we denote this as $X \cong Y$.

Next, we introduce \textbf{affine independence}, a generalisation of the more familiar concept of linear independence to point sets. A set of $k+1$ points $v_0, ..., v_k$ is said to be affinely independent if there exist real numbers not $a_0, ..., a_k$ not all zero such that $\sum_{i=0}^d a_iv_i = 0$ and $\sum_{i=0}^k a_i = 0$. We can express this in a form more familiar to students of linear algebra, in that the $k$-vectors $v_k - v_0, ..., v_1 - v_0$ are linearly independent. Therefore, a set of 2 distinct points is automatically affinely independent, a set of 3 distinct points are affinely independent if they do not lie on a common line, etc. 

Now we will introduce the \textbf{simplex}, which generalises the notion of a triangle into arbitrary dimensions. Formally, a simplex is defined as the convex hull of a finite affinely independent point set. Thus, a simplex in $\mathds{R}^1$ is a line segment as this forms the most "efficient" possible connection between any 2 points. Similarly, a simplex in $\mathds{R}^2$ is a triangle as no other shape can be drawn that encompasses all connections between 3 affinely independent points with fewer line segments. Extensions into higher-dimensional spaces can be made similarly. A face of a simplex is the convex hull of an arbitrary subset of vertices of the simplex, which forms a lower-dimensional simplex of its own right by definition. 

A \textbf{simplicial complex} is defined as a set of simplicies satisfying the properties that each face of a simplex is included within the set and that every intersection between 2 simplices must be a face of both. Below are some pictured examples of some families of simplices that satisfy and violate the conditions:
\includegraphics[width=12cm, height=3cm]{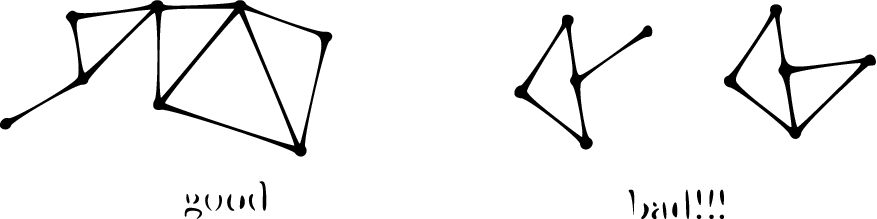}
    \begin{center}
        Credit: Matoušek
    \end{center}

The common understanding of a polyhedra can thus be represented a union of a simplicial complex and is denoted $||\triangle||$. 

A \textbf{triangulation} of any topological space $X$ is a simplicial complex $\triangle$ such that $X \cong ||\triangle||$. This results in a simplicial complex that can be "smoothed" back into the original shape through continuous deformation.

\includegraphics[width=8cm, height=4cm]{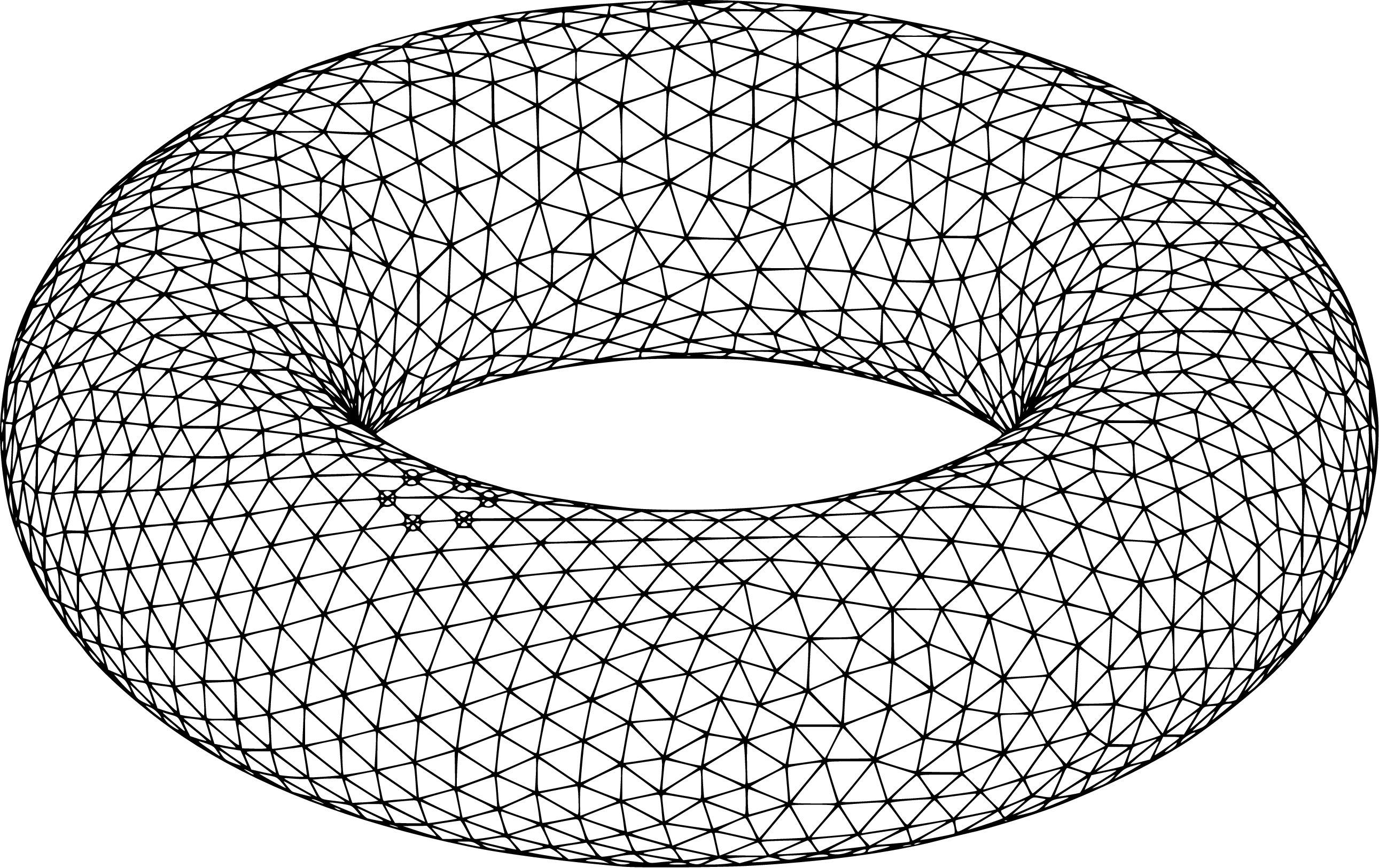}
    \begin{center}
    The triangulation of a torus. Credit: https://math.stackexchange.com/
    \end{center}

Having introduced the necessary topological background, we now proceed to stating the Borsuk-Ulam Theorem, which will be of keen interest as we seek to prove theorems in discrete mathematics.

\section{The Borsuk-Ulam Theorem and Equivalent Results}
Let $S^n$ denote the unit sphere in $\mathds{R}^{n-1}$. That is, $S^1$ is a circle, $S^2$ is a traditional sphere, and so on. The Borsuk-Ulam Theorem states the following:
\newtheorem*{borsuk}{Borsuk-Ulam Theorem}

\begin{borsuk}
    For every continuous mapping $f: S^n \rightarrow \mathds{R}^n$ there exists a point $x \in S^n$ with $f(x) = f(-x)$.
\end{borsuk}
The points $x, -x$ are referred to as antipodal points. The canonical example of this phenomena is that there exist 2 antipodal points on the globe where temperature and barometric position are equal (assuming temperature and barometric pressure to be continuous maps).

In order to aid the transition to a combinatorial / graph-theoretic viewpoint, we now present 2 equivalent formulations of the theorem in the language of discrete mathematics.

\newtheorem*{tucker}{Tucker's Lemma}
\newtheorem*{lst}{Lusternik–Schnirelmann Theorem}

\begin{tucker}
    Let $T$ be a triangulation of $B^n$ that is antipodally symmetric on the boundary. Let $\lambda: V(T) \rightarrow \{+1, -1, +2, -2,..., +n, -n\}$ be a labeling of the vertices of $T$ that satisfies $\lambda(-v) = -\lambda(v)$ for every vertex $v \in \partial B^n$ (that is, $\lambda$ is antipodal on the boundary). Then there exists a 1-simplex (an edge) in $T$ that is complementary; i.e., its two vertices are labeled by opposite numbers.
\end{tucker}

\begin{lst}
    Whenever $S^n$ is covered by $n+1$ sets $A_1, A_2, ... ,A_{n+1}$, each $A_i$ open or closed, there is an $i$ such that $x, -x \in A_i$.
\end{lst}



\section{The Ham Sandwich Theorem}
\newtheorem*{hstm}{Ham Sandwich Theorem for Measures}

\begin{hstm}
    Let $\mu_1, \mu_2, ..., \mu_d$ be finite Borel measures
    on $\mathds{R}^d$ such that every hyperplane has measure $0$ for each of the $\mu_i$. Then there exists a hyperplane $h$ such that $\mu_i(h^+) = \frac{1}{2} \mu_i(\mathds{R}^d)$ for $i = 1, 2, ..., d$, where $h^+$ denotes one of the half-spaces defined by $h$.
\end{hstm}
The theorem derives its informal name from the case $n = 3$, 
where the Borel measures in questions take the form of 2 pieces of bread and a slice of ham. In this case, the ham sandwich theorem guarantees that there exists a planar cut with a knife that divides both pieces of bread and ham precisely in two.

\begin{proof}[\rm\bf{Proof}] First, let us consider $\textbf{u} = (u_0,u_1,...,u_d)$ a point on sphere $S^d$. If not all of the components in $u$ are zeros, we define $$h^+(u) := \{(x_1,...,x_d) \in \mathbb{R}^d : u_1x_1 +...+u_dx_d \leq u_0\} $$
Consider the half-space assigned to the antipodal point $\textbf{-u} = (-u_0,-u_1,...,-u_d)$, as defined $$h^+(-u) := \{(x_1,...,x_d) \in \mathbb{R}^d : u_1x_1 +...+ u_dx_d \geq u_0\} $$
We now choose u = (1,0,...,0) and -u = (-1,0,...,0), so we have $$h^+(u) := \{(x_1,...,x_d) \in \mathbb{R}^d : 0 \leq 1\} = \mathbb{R}^d  (*)$$
$$h^+(-u) := \{(x_1,...,x_d) \in R^d : 0 \geq 1\} = \emptyset  (**)$$
We now define a new family of functions $ f_i: S^d \rightarrow \mathbb{R}^d$ by $$ f_i(u) := \mu_i(h^+(u))$$
If we want to apply to the Borsuk-Ulam Theorem, we need to check the antipodal mapping first. Let $f(u_0) = f(-u_0)$ for some $u_0 \in S^d$, then by (*) and (**), we know that $h^+(u_0)$ is a half-space. As defined, every hyperplane has measure 0, so there exist $u \in S^d$ such that $f(u) = 0$, thereby satisfying the antipodal mapping.
Next, we will check the continuity of $f$. Let $(u_n)^\infty_{n=1}$ be a sequence of points of $S^d$ converging to u. Our goal is to prove $\mu_i(h^+(u_n))\rightarrow \mu_i(h^+(u))$. By assumption, if $x \not\in \partial h^+(u)$, then $x \in h^+(u_n)$ if and only if $x \in h^+(u)$ for all n sufficiently large. Now define the characteristic function of $h^+(u)$ $$g(x) = 1, x \in h^+(u)$$ $$g(x) = 0, x \not\in h^+(u)$$ \\ Similarly, define the characteristic function of $h^+(u_n)$ $$g_n(x) = 1, x \in h^+(u_n)$$ $$g_n(x) = 0, x \not\in h^+(u_n)$$ \\
So, $g_n(x) \rightarrow g(x)$ for all $x \not\in \partial h^+(u)$. By assumption, $\mu_i(\partial (h^+(u)) = 0$, so (add reference), $\mu_i(\partial (h^+(u_n)) = \int g_n d\mu_i \rightarrow \int g d\mu_i = \mu_i(\partial (h^+(u)) $
\end{proof}

\newtheorem*{hstp}{Ham Sandwich Theorem for Point Sets}

\begin{hstp}
    Let $A_1, A_2,..., A_d \subset \mathbb{R}^d$ be finite point sets. Then there exists a hyperplane $h$ that simultaneously bisects $A_1, A_2,...,A_d$.
\end{hstp}

Note: If a hyperplane $h$ bisects $A_i$ with $2k+1$ points, then each open half-space defined by $h$ contains at most $k$ points, which means at least one point lies on the hyperplane $h$. 

\begin{proof}[\rm\bf{Proof}]
The general idea of this proof is to replace the points sets with small balls then apply the Ham Sandwich Theorem for measures. In order to do this, we will introduce the following definition: 

\begin{definition}
      Finite point sets $A_1, A_2,..., A_d$ are in \textbf{general position} in $\mathds{R}^d$ if each distinct $A_i$ disjoint and no more than $d$ points of any set lie on a common hyperplane. 
\end{definition}

\includegraphics[width=6cm, height=6cm]{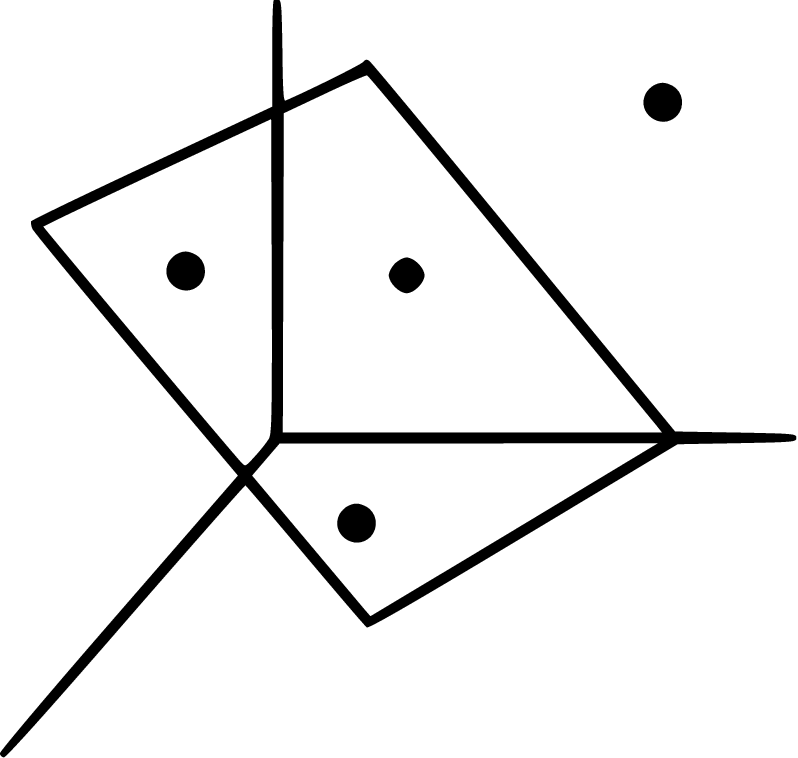}
    \begin{center}
    A point set in general position in $\mathds{R}^3$.
    \end{center}

First, we suppose each $A_i$ has odd number of points and $A_1\dot\cup A_2\dot\cup ... \dot\cup A_d$ is in general position. Now let $A_i^\epsilon$ be balls sets by replacing each point in $A_i$ with a solid ball of radius $\epsilon$ that is small enough such that no hyperplane can intersect more than d balls of $\cup A_i^\epsilon$. By the Ham Sandwich Theorem for measures, we can find a hyperplane $h$ bisecting the sets $A_i^\epsilon$ at the same time. Since each $A_i$ has odd cardinality, then at least one point must lie on the hyperplane $h$. Since no more than $d$ balls can be intersected simultaneously, then $h$ intersects only one ball of each $A_i^\epsilon$. It means $h$ passes through the center of this ball, thereby $h$ bisects each $A_i$.\par
Second, we still suppose that each $A_i$ has odd number of points but the position of $A_1\dot\cup A_2\dot\cup ... \dot\cup A_d$ can be random. We use a perturbation argument to construct them in general position. For every $\eta > 0$, let ${A_i}_\eta$ be point sets by moving each point in $A_i$ by at most $\eta$ so that ${A_1}_\eta \dot\cup {A_2}_\eta \dot\cup ... \dot\cup {A_d}_\eta$ is in general position. So we can find a hyperplane $h_\eta$ bisecting ${A_i}_\eta$. So define $$ h_\eta = \{x \in \mathbb{R}^d : \langle a_\eta,x \rangle = b_\eta\} $$ where $a_\eta$ is a unit vector, so $b_\eta$ lies in a bounded interval for $\eta$ small enough. By the Bolzano-Weierstrass Theorem, as $\eta \rightarrow 0$, for every pair $(a_\eta,b_\eta)$ there exists a limit point $(a,b) \in \mathbb{R}^{d+1}$ Now, let $h$ be the hyperplane obtained from $\langle a, x\rangle = b$. Next, let $\eta_1 > \eta_2 > ... $ be a sequence converging to 0 satisfying $(a_{\eta_j},b_{\eta_j}) \rightarrow (a,b)$. Consider a point $x$ is $\delta > 0$ away from hyperplane $h$, then for $j$ sufficiently large, it is at least $\frac{1}{2}\delta$ away from $h_{\eta_j}$. Now, consider that we have $k$ points of $A_i$ in one open half-space generated by $h$, then for all $j$ sufficiently large, the corresponding open half-space generated by $h_{\eta_j}$ contains at least $k$ points of $A_{i,\eta_j}$. Therefore, we still achieve a bisection of $A_i$ without the general position in our assumption.\par
Third, let some of the $A_i$ have an even cardinality. By deleting and adding one random point from each even-sized $A_i$, we can still check the bisection. 
\end{proof}

\newtheorem*{hstg}{Ham Sandwich Theorem (General Position Version)}

\begin{hstg}
    Let $A_1, A_2,..., A_d \\ \subset \mathbb{R}^d$ be disjoint finite point sets in general position. Then there exists a hyperplane h that bisects each $A_i$, such that there are exactly $\lfloor \frac{1}{2} |A_i| \rfloor$ points from $A_i$ in each of the open half-spaces defined by $h$, and at most one point of $A_i$ on the hyperplane $h$.
\end{hstg}

\begin{proof}[\rm\bf{Proof}] By \textbf{Ham Sandwich Theorem for point sets}, we choose a random hyperplane $h$ such that simultaneously bisects $A_1,A_2,...,A_d$. We now reassign our coordinate system to let $h$ be the horizontal hyperplane $x_d = 0$. We now define $$ B:= h\cap (A_1 \cup ... \cup A_d).$$ We claim that $B$ consists of at most $d$ affinely independent points. By assumption, $A_i$ is in general position, so there are at most d points on $B$. Suppose there exists $x$ points that are affinely dependent, then the hyperplane will require at least $d-x+1$ points to generate, which is a violation of the general position assumption. 

We now add $d-|B|$ points to $B$ to get a $C \subset h$ with d affinely independent points. Since the points of $C$ are affinely independent, then for each $a \in C$, we can determine another $a'$ such that $a'=a$, or $a' = a+ \epsilon e_d$, or $a' = a- \epsilon e_d$. We let $h' = h'(\epsilon)$ be the desired hyperplane.
\end{proof}

\section{On Multicolored Partitions and Necklaces}
We begin with an easy result to establish the relevance of the ham sandwich theorem in multicolored partition problems.

\newtheorem*{thm}{Theorem}

\begin{thm}
 Consider sets $A_1, A_2,...,A_d$, of $n$ points each, in general position in $\mathds{R}^d$; imagine that the points of $A_1$ are red, the points of $A_2$ blue, etc. (each $A_i$ has its own color). Then the points of the union $A_1\cup ... \cup A_d$ can be partitioned into “rainbow” d-tuples (each d-tuple contains one point of each color) with disjoint convex hulls.
\end{thm}

For a visual understanding of the theorem, consider the below illustration: 
\includegraphics[width=5cm, height=4cm]{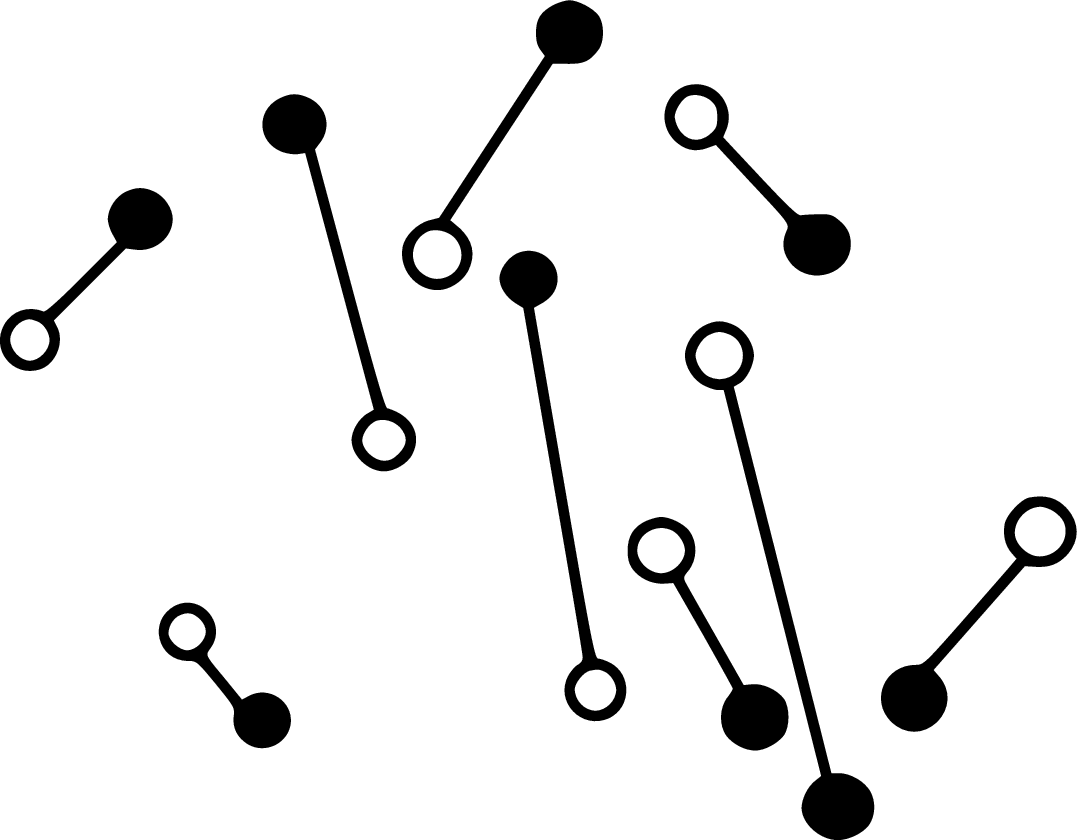}

\begin{proof}[\rm\bf{Proof}]:
First, let us consider the case where $n$ is an even number. By the general position version of the ham sandwich theorem, there exists a hyperplane which bisects all the $A_i$ with no points lying on the hyperplane. For $n$ odd, there exists a hyperplane that simultaneously bisects all $d$ point sets such that exactly one point from each set lies on the plane by the ham sandwich theorem for point sets in general position. We will let the convex hull connecting the points on this plane be the first of the rainbow $d$-tuples. We then recursively split the points remaining on both half-planes formed on either side of the initial hyperplane, according to the following algorithm:

- If the number of points in each half-plane, divide these points evenly in half with a hyperplane by the ham sandwich theorem.

- If the number of points in each half-plane is odd, apply the previously mentioned version of the ham sandwich theorem, such that one point of each $A_i$ remains on the plane. This will form another rainbow $d$-tuple.

Eventually, there will be exactly $d$ points of different colors lying on each side of the hyperplane, which we can place onto a convex hull. Thus, every point in the union of point sets now lies on a disjoint rainbow convex hull.
\end{proof}

We now proceed to tackling the necklace-splitting problem. 

\newtheorem*{problem}{The Necklace-Splitting Problem}

\begin{problem}
     Infamous jewel thieves G. Huang and H. Yan steal an invaluable necklace. The necklace is so valuable not only because of the precious gems it contains, but because they are embedded in solid gold. Huang and Yan want to split the jewels evenly between themselves while wasting as little gold as possible. In how few cuts is this possible?
\end{problem}

Luckily, Huang and Yan are accomplished mathematicians and are familiar with the following theorem:

\newtheorem*{necklace}{Necklace Theorem}

\begin{necklace}
     Every (open) necklace with $d$ kinds of stones can be divided between two thieves using no more than $d$ cuts.
\end{necklace}

Before we begin our proof, we require the following definition and lemma.

\begin{definition}
     A \textit{moment curve} ($\gamma$) is the curve in $\mathds{R}^d$ given by the parametric equation $(t, t^2, t^3, ..., t^d)$
\end{definition}

\newtheorem*{lemma}{Lemma}

\begin{lemma}
     No hyperplane intersects the moment curve $\gamma$ in $\mathds{R}^d$ in more than $d$ points
\end{lemma}

\begin{proof}[\rm\bf{Proof of Lemma}]
A hyperplane $h$ has an equation $a_1x_1 + a_2x_2 + ··· + a_dx_d = b$ with $(a_1, a_2, ..., a_d)$. To find all intersections between $h$ and $\gamma$, we substitute $\gamma$ in for $h$ and receive $a_1t + a_2t^2 + ... + a_dt^d = b$. Therefore, the values of $h$ at which $\gamma$ intersects are the roots of the equation $0 = a_1t + a_2t^2 + ... + a_dt^d$, which has no more than $d$ solutions.
\end{proof}

We may now proceed to proving the Necklace Theorem.

\begin{proof}[\rm\bf{Proof of Theorem}]
We place our necklace in $\mathds{R}^d$ along the moment curve $\gamma$. We define a system of sets $A_i$, where the points of $A_i$ are stones of the $i$-th kind. By the discrete general position version of the ham sandwich theorem, there exists a hyperplane $h$ that simultaneously bisects each $A_i$. By Lemma 0, $h$ cuts $\gamma$ (and thus the necklace) at no more than $d$ points. Assuming all sets $A_i$ have the same size, no stones lie on $h$, and thus we have an even division.
\end{proof}

There exists a second topologically-inspired proof for the Necklace Theorem. Before tackling this proof, we require the use of the following lemma, a continuous version of the necklace theorem:

\newtheorem*{hobby}{Lemma (Hobby-Rice Theorem)}

\begin{hobby}
    Let $\mu_1$, $\mu_2$, ..., $\mu_d$ be continuous probability measures on $[0, 1]$. Then there exists a partition of $[0, 1]$ into $d + 1$ intervals $I_0, I_1, ..., I_d$ (using $d$ cut points) and signs $\epsilon_1, \epsilon_2, ..., \epsilon_d \in \{ -1, +1 \}$ with $\sum_{j=0}^d \epsilon_j \cdot \mu_i(I_j) = 0$ for $i = 1,2, ..., d$.
\end{hobby}

\begin{proof}[\rm\bf{Proof}]
For every point $\textbf{x} = (x_1, x_2, ..., x_d)$ of the sphere $S^d$, we associate a division of the interval $[0,1]$ into $d + 1$ parts, with lengths $x_1^2, x_2^2, ..., x_d^2$. For example, the point $(1, 0, 0)$ on the sphere in $\mathds{R}^d$ is associated with the unbroken interval $[0, 1]$. Specifically, each cut on the interval is made at the point $z_i$, where $z_i = x_1^2 + ... + x_i^2$ for any $i$ between $0$ and $d+1$, inclusive. The sign $\epsilon_j$ for the interval $I_j = [z_{j-1}, z_j]$ is chosen as the sign of the coordinate $x_j$. From this, we may define a function $g: S^d \rightarrow \mathds{R}^d$, where $g(x) = \sum_{j=1}^{d+1}$ sign$(x_j) \cdot \mu_i([z_{j-1}, z_j])$. We can interpret $g$ as allocating the measure of the intervals associated with positive sign to the first thief, and the measure of the intervals associated with the negative sign to the second thief. $g$ is continuous, since each $\mu_i$ is a continuous probability measure. $g$ is also antipodal, as multiplying $x$ by $-1$ reverses the sign of the coordinate $x_j$ while the probability measure $\mu_i([z_{j-1}, z_j])$ remains the same, such that $g(-x) = -g(x)$. Therefore, by the Borsuk-Ulam theorem, there exists an $\textbf{x}$ such that $g(\textbf{x}) = 0$. This implies that the amount of measure allocated to the first thief is exactly the same as the amount allocated to the second thief. Hence, there exists a just division under this $\textbf{x}$.
\end{proof}

We may now proceed to proving the Necklace Theorem (again).

\begin{proof}[\rm\bf{Proof}]
\includegraphics[width=12cm,height=3cm]{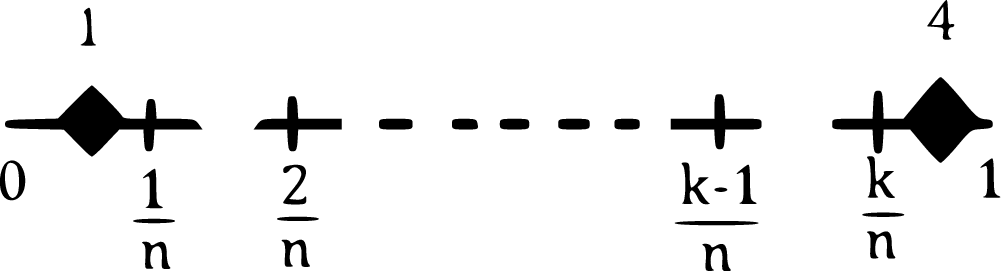}\par
Firstly, we place the necklace on the interval $[0, 1]$, such that the $k$-th stone corresponds to the interval $[\frac{k-1}{n}, \frac{k}{n})$, as pictured below.

Now we define the family of functions $f_i$ such that $f_i: [0,1] \rightarrow [0,1]$,

    \begin{equation*}
    f_i(x)= \left \{
        \begin{array}{ll}
            1 & \text{ if } k \text{-th stone of the necklace is of } i \text{-th kind}\\
            0 & \text{ otherwise.}
        \end{array}
    \right.
    \end{equation*}




With this knowledge, we now can define a family of measures $\mu_i$ describing the fraction of stones of the $i$-th kind that is on any part A of the necklace:

$\mu_i(A) = \frac{n}{t_i} \cdot \int_A f_i(x) \cdot dx$

where $t_i$ represents the total number of stones of type $i$. To see why this formula makes sense, evaluate the integral $\int_A f_i(x) \cdot dx$ to get $1 |_A$. Remember that we placed each stone $k$ within the interval $[\frac{k-1}{n}, \frac{k}{n})$, which has length $\frac{1}{n}$. Therefore, $1 |_A$ evaluates to $\frac{1}{n} \cdot t_i$, the fraction of stones of type $i$ that are within A. Therefore, we must multiply the entire integral by $\frac{n}{t_i}$ so that it integrates to unity, which is the measure of the complete interval $[0, 1]$. We now apply the Hobby-Rice Theorem, in the case where $\epsilon_j = (-1)^j$ so that $\sum_{j=0}^d (-1)^j \cdot \mu_i(I_j) = 0$. Similar to the above proof, we can assign the positive intervals to the first thief, and the negative intervals to the second thief. While this division is fair, some stones may be split in two between thieves, which violates the spirit of the problem. To solve this, realize that if a cut splits a stone in two, then there must exist another cut subdividing a stone of type $i$ to balance the measure. In this case, we can move both cuts away from the stones without altering the balance.
\end{proof}

\section{Kneser's Conjecture and the Lovasz-Kneser Theorem}
After the application of the Borsuk-Ulam Theorem in the necklace problem we proceed to the field of combinatorics. In this section, the Borsuk-Ulam Theorem will only serve as inspiration, and the proof will be primarily from the fields of combinatorics and graph theory. Therefore, we need some basic combinatorial background before starting the proof.

Consider a set $[n]$ with $n$ elements, and another number $k$ such that $n \geq 2k-1$. Denote the set $[n]\choose k$ as the subset of $[n]$'s power set where every element in $[n]\choose k$ is a subset of $[n]$ with exactly $k$ elements. Here we have a mapping $f: {[n]\choose k}\rightarrow M$, where $M$ is a set of distinct $i$ colors ($M$: \{$m_1$,...,$m_i$\}). The mapping $f$ satisfies this condition:

\begin{quote}
    For any $K_1, K_2 \in {[n]\choose k}$, if $K_1 \bigcap K_2 = \varnothing$, then $f(K_1) \neq f(K_2)$
\end{quote}

Of course, there are more than one mapping like this f coloring $k$-element subsets of $[n]$ with the principle above. One of the possible ways for such coloring is to color each $k$-element subset with a distinct color, and this way of coloring needs $n\choose k$ (number of $k$-element subsets in $[n]$) different colors. However, the Lovasz-Kneser Theorem is not interested in having as many colors as possible: it focuses on the minimum number of distinct colors we need in order to color $k$-element subsets of $[n]$ satisfying the condition. Later, we will refer to any coloring of sets in $[n]\choose k$ satisfying the condition as a 'proper coloring'.

\newtheorem*{kneser}{Kneser Conjecture}

\begin{kneser}
     The minimum number of color needed for a proper coloring of elements in $[n]\choose k$ is $n-2k+2$.
\end{kneser}

To interpret this theorem we introduce some background from graph theory. Consider a graph in which each element in $[n]\choose k$ is a vertex. For any two dots in the graph representing two elements $K_1$ $K_2$ in $[n]\choose k$, if the two subsets don't share any element of $[n]$, then there is an edge connecting the two dots representing $K_1$ $K_2$. Here, a proper coloring refers to a pattern of coloring each dot in the graph such that, any two vertices connected in the graph are colored by two distinct colors. 

For a set X and its set system $\mathcal{F}$, the graph $G$ that represents each element of $\mathcal{F}$ (each subset of X in $\mathcal{F}$) as a vertex is called the Kneser Graph of $\mathcal{F}$. The chromatic number of this graph, denoted by $\chi (G)$, is the minimum number of $k$ where $k$ distinct colors are able to color the graph $G$ properly (two vertices connected by an edge are colored by two different colors). The Kneser graph of set system $[n]\choose k$ representing each element in $[n]\choose k$ as a vertex is denoted as $KG_{n,k}$. Thus, the theorem claims that, the chromatic number of $KG_{n,k}$, $\chi (KG_{n,k})$, is equal to $n-2k+2$.

For example, when we have a set $[5]$ with five elements and the set system consists of all subsets of $[5]$ with only two elements, the Kneser Graph looks like this, where every dot represents a subset of $[5]$ with exactly two elements.
\includegraphics[width=15cm, height=9cm]{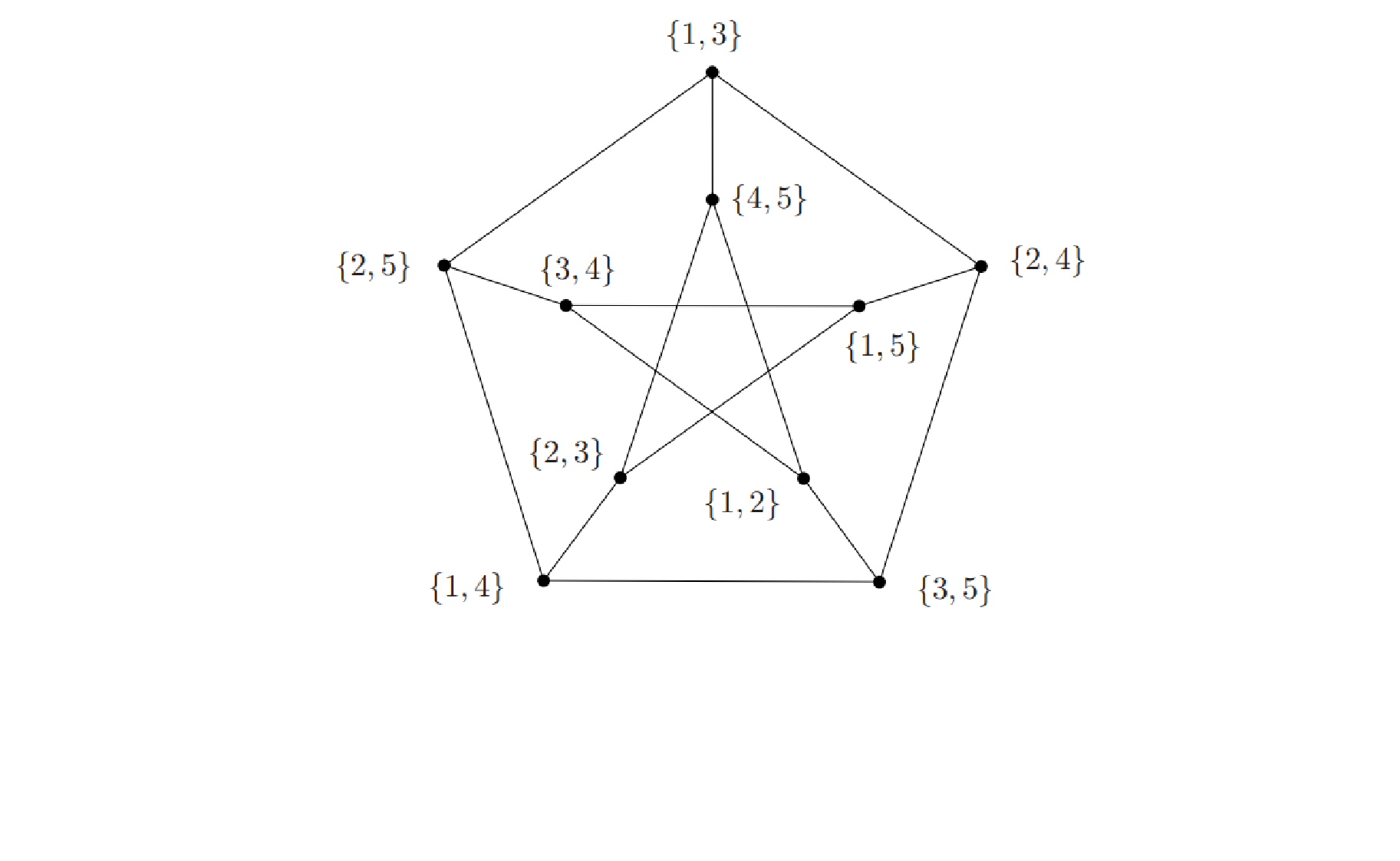}
    \begin{center}
        Credit: Matoušek
    \end{center}

\newtheorem*{lkneser}{The Lovasz-Kneser Theorem (reformulated Kneser Conjecture)}

\begin{lkneser}
     For natural number $k$ and $n\geq2k-1$,  $\chi (KG_{n,k})$, the chromatic number of $KG_{n,k}$ is equal to $n-2k+2$.
\end{lkneser}

\begin{proof}[\rm\bf{Proof}]
To prove this theorem we need to show that $\chi (KG_{n,k})$ is both smaller or equal to $n-2k+2$ and greater or equal to $n-2k+2$. Then we can conclude  $\chi (KG_{n,k})$ must be equal to $n-2k+2$.

Firstly we show $\chi (KG_{n,k})$ is smaller or equal to $n-2k+2$. Here we view the set $[n]$ identically to the set $\{1,...,n\}$ where every element in $[n]$ is assigned to a number. Consider the $n-2k+2$ distinct color set \{$m_1$,...,$m_{n-2k+2}$\}, for each vertex in the Kneser Graph $KG_{n,k}$, color it with the $i^{th}$ color in \{$m_1$,...,$m_{n-2k+2}$\}, where $i$ = $\text{min}\{\text{min}(F),n-2k+2\}$. Here $min(F)$ denotes the smallest number among each number assigned to elements in set $F$. Now, each vertex in $KG_{n,k}$ is assigned to a color in \{$m_1$,...,$m_{n-2k+2}$\}. We need to verify that this coloring pattern is a proper coloring, which means, if two vertices in $KG_{n,k}$ are colored in the same color then they must share some elements in $[n]$. 

Here, if two vertices are colored by the same color labelled as number $j$, and $j\textless n-2k+2$, this means the $k$-element subsets of $[n]$ represented by the two vertices share the same element of $[n]$ such that the shared element is labelled as $j \in \{1,...,n\}$. If two vertices are color by the same color that is labelled as number $n-2k+2$, then each of the vertices represents a $k$-element subset of elements of $[n]$ labelled as numbers in \{$n-2k+2$,...,$n$\}. Notice that, the subset of elements of $[n]$ labelled as numbers in \{$n-2k+2$,...,n\} only has $2k-1$ distinct elements. Then for the two $k$-element subsets, they must be sharing at least one element of $[n]$ labelled as numbers in \{$n-2k+2$,...,n\}.

For this coloring pattern, any two vertices colored in the same color represents two $k$-elements subsets that are not disjoint. Thus, $n-2k+2$ is sufficient for a proper coloring of $KG_{n,k}$. By definition of $\chi (KG_{n,k})$, it must be smaller or equal to $n-2k+2$. 

The first half of the proof focuses exclusively on set theory, and the application of the Borsuk-Ulam Theorem will show up in the second half of the proof where we confirm $\chi (KG_{n,k}) \geq n-2k+2$.

To show that $\chi (KG_{n,k}) \geq n-2k+2$, we can use an alternative approach by showing that for any $i\leq n-2k+2$, $i$ distinct colors are insufficient for a proper coloring of the Kneser graph $KG_{n,k}$. More directly, we only need to show that $n-2k+1$ distinct colors can't form a proper coloring for $KG_{n,k}$ (If $n-2k+1$ distinct colors can only form a coloring pattern where some pairs of vertices joined by an edge is colored with the same color, then a color set with less distinct colors can only have even more joint pairs colored by the same color.).

Consider the color set \{$m_1$,...,$m_d$\}, where $d = n-2k+1$. We proceed to arguing by contradiction through assuming there is a proper coloring of $KG_{n,k}$ by $n-2k+1$ distinct colors. Take $X \subset S^{d}$ where $X$ is a $n$-point subset of $S^{d}$. One additional requirement for $X$ is that, for any hyperplane in $R^{d+1}$ passing through the origin, it can't contain more than d points of set $X$. This is a requirement based on the general position principle (For $R^{d+1}$, a set is in general position implies that each hyperplane can have at most $d$ points from the set, here the requirement changes $d$ to $d-1$). It is easy to assign each element in set $[n]$ to a point in $X$ and such assignment is a one-to-one correspondence. Thus, we no longer the vertex set of $KG_{n,k}$ as $[n]\choose k$, instead we regard it as $X\choose k$ where every vertex represents a subset of $X$ with $k$ points on $S^{d}$. (Every point of $X$ corresponds to an element in $[n]$.

Now, we define subsets $A_{1},...,A_{d} \subseteq S^{d}$ by this way: for $x \in S^{d}$, if the open hemisphere centered at $x$ contains a $k$-element subset of $X$ such that, the vertex in $KG_{n,k}$ representing this subset is colored as the $i$-th color in the color set \{$m_1$,...,$m_d$\}. In addition to $A_{1},...,A_{d} \subseteq S^{d}$, we define $A_{d+1}$ as $S^{d} \backslash A_{1} \bigcup ... \bigcup A_{d}$. It is evident that $A_{1},...,A_{d+1}$ forms a cover of $S^{d}$. By the Lyusternik–Shnirel’man theorem derived from the Borsuk-Ulam theorem, given a cover of $A_{1},...,A_{d+1}$ of $S^{d}$, there exists a set $A_{l}$ among the sets forming the cover such that $A_{l}$ contains a pair of antipodal points. Denote the pair of antipodal points as $y$, $-y$. 

Consider ${l} \in \{$1$,...,$d$\}$, then the open hemisphere centered at $y$ contains a $k$-element subset of $X$ such that the vertex in $KG_{n,k}$ representing this subset is colored in color $l$, and the open hemisphere centered at $-y$ contains another $k$-element subset such that the the vertex in $KG_{n,k}$ representing this subset is colored in the way. Notice that the two $k$-element subsets are in different open hemispheres, so they must be disjoint. Therefore such coloring must have violated the principle of proper coloring, and ${l}$ must not be in \{$1$,...,$d$\}.

The only possibility left is ${l} \in \{$d+1$\}$, which implies $A_{d+1}$ contains a pair of antipodal points $y$,$-y$. By the way we define $A_{1}$,...,$A_{d}$,the open hemisphere $H(y)$ centered at $y$ can't contain any $k$-element subset of $X$ (If it does then the vertex representing this subset will be colored with a color and $y$ will no longer be in $A_{d+1}$). Thus the open hemisphere centered at $y$ can have at most $k-1$ points from $X$, same for the open hemisphere $H(-y)$ centered at $-y$. Given set $X$ on $S^{d}$ has $n$ elements, there must be at least $n-(2k-2)$ points on $S^{d} \backslash (H(y) \bigcup H(-y))$. Notice that $S^{d} \backslash (H(y) \bigcup H(-y))$ is an equator of $S^{d}$ which lies on a hyperplane passing through the origin. By our setup, no hyperplane passing through the origin can't contain more than d points of $X$. However, here the hyperplane which $S^{d} \backslash (H(y) \bigcup H(-y))$ lies on contains all points of $X$ on the equator, and there are $n-2k+2 = d+1$ points on the equator (Recall $d = n-2k+1$). Thus, ${l}$ must not be equal to $d+1$ too.

\includegraphics[width=8cm, height=8cm]{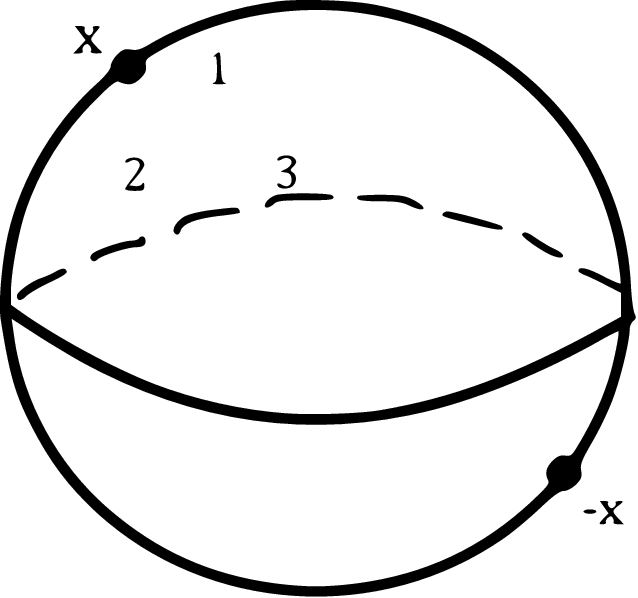}
    \begin{center}
        Here, $x$ is in the $A_{red}$ point set.
    \end{center}

By our assumption that there is a proper coloring of $KG_{n,k}$ by $n-2k+1$ distinct colors, there is no set containing a pair of antipodal points among $A_{1},...,A_{d+1}$ which form a cover of $S^{d}$. This contradicts the Borsuk-Ulam Theorem which is already accepted as true. Therefore the assumption must be false and $n-2k+1$ distinct colors are insufficient for a proper coloring of $KG_{n,k}$. So, at least $n-2k+2$ is needed for a proper coloring of $KG_{n,k}$. So we proved$\chi (KG_{n,k}) \geq n-2k+2$.

Given $\chi (KG_{n,k}) \geq n-2k+2$ and $\chi (KG_{n,k}) \leq n-2k+2$, we proved that $\chi (KG_{n,k}) = n-2k+2$, the Lovasz-Kneser Theorem is proved.
\end{proof}

\section{Dol'nikov's Theorem}
Dol'nikov's theorem is a generalization of the Lovász-Kneser theorem. The difference is that instead of giving an exact number for $\chi(KG_{n,k})$, Dol'nikov's theorem provides a lower bound for $\chi(KG(\mathcal{F}))$ where $\mathcal{F}$ is an arbitrary (finite) set system. However, to properly define this lower bound, we must first define hypergraph and what it means for a hypergraph to be $m$-colorable.\par
\begin{definition}
    Given any set $X$ and $\mathcal{F} \subseteq 2^X$ a system of subsets of $X$, we call $\mathcal{F}$ a set of hyperedges and the pair $(X, \mathcal{F})$ a hypergraph generated by $X$ and $ \mathcal{F}$.
\end{definition}  
\begin{definition}
    The hypergraph is $m$-colorable if there exists a coloring $c: X \rightarrow [m]$ such that no hyperedge is monochromatic under $c$, i.e. $|c(F)|>1$ for all $ F \in \mathcal{F}$.
\end{definition} Notice that here we have identified the set of $m$ colors with the set of integers $[m]=1,2,...,m$. Also, one should be careful to distinguish this coloring from the proper coloring of a Kneser Graph. \par
Now, given a hypergraph $(X,\mathcal{F})$, we define the $m$-colorability defect as follows:
\begin{equation*}
    cd_m(\mathcal{F})=\text{min}\{|Y|:Y \subseteq X \text{ and }(X\backslash Y,\{F \in \mathcal{F}:F \cap Y=\varnothing \}) \text{ is m-colorable}\}
\end{equation*}
where $|Y|$ denotes the number of elements in $Y$. \par
(Note that we can view the same collection of points as being either a Kneser graph, with the grouping of certain points being nodes of the Kneser graph, or as a hypergraph, with the nodes of the Kneser graph being hyperedges on the hypergraph, grouping together nodes).

In particular, we are concerned with $cd_2(\mathcal{F})$, which can be interpreted as the minimum number of points of $X$ being colored white such that when the rest of the points in $X$ are colored red or blue, no $F \in \mathcal{F}$ is completely red or completely blue. Based on these definitions, Dol'nikov's theorem can be formally stated as follows:

\newtheorem*{Dol}{Dol'nikov's Theorem}

\begin{Dol}
For any hypergraph $(X,\mathcal{F})$, we have \\
\centerline{$\chi(KG(\mathcal{F})) \geq cd_2(\mathcal{F})$}. 
\end{Dol}
Next, we supply 2 proofs of the theorem, both involving Borsuk-Ulam in some way.
\let\oldproofname=\proofname
\renewcommand{\proofname}{\rm\bf{\oldproofname}}
\begin{proof}
    This section is strikingly similar to the latter half of the proof of the Lovász-Kneser theorem. Let $d=\chi (KG(\mathcal{F}))$. We then identify the points of $X$ with a set of points in general position on $S^d$ and $\mathcal{F}$ with the same set system based on the point set on $S^d$. We construct the sets $A_1, ..., A_d$ by the following process: $\forall x \in S^d$, assign $x\in A_i$ if the open hemisphere $H(x)$ contains a set $F \in \mathcal{F}$ colored by color $i$. In other words, $A_i=\{x \in S^d: H(x) \text{ contains an }F\in \mathcal{F}\text{ of color } i\}$. Furthermore, we define $A_{d+1}=S^d\backslash\{A_1 \cup ... \cup A_d\}$.\par
    Notice that $S^d=A_1 \cup A_2 \cup ... \cup A_{d+1}$. Then by the Lyusternik-Shnirel'man theorem, we have $x, -x \in A_i$ for some $A_i$. \par
    Suppose $i<d+1$, then since $x\in A_i$, $H(x)$ contains some $F \in \mathcal{F}$ with color $i$ and $H(-x)$ also contains some $F' \in \mathcal{F}$ with the same color. Since $H(x)$ and $H(-x)$ are disjoint, $F \subset H(x)$ and $F' \subset H(-x)$ must also be disjoint (implying they are connected in the Kneser Graph). Now we have an edge between two nodes with the same color $i$, meaning the coloring of the Kneser Graph is not proper. However, we assumed proper coloring by letting $d=\chi (KG(\mathcal{F}))$. Contradiction.\par
    Hence $i=d+1$. In this case, by definition of $A_{d+1}$, neither $H(x)$ nor $H(-x)$ can contain any $F \in \mathcal{F}$, i.e. any $F \in \mathcal{F}$ either completely consists of points from both hemispheres or it contains points on the equator. If we color $Y$=the points of $X$ that are on the equator (which is at most $d$ points by the general position assumption) white, those in $H(x)$ red, and those in $H(-x)$ blue, then any $F\in \mathcal{F}$ will have at least 2 colors. Equivalently speaking, $\{F \in \mathcal{F}:F\cap Y=\varnothing\}$ will be exactly those $F \in \mathcal{F}$ consisting of points from both hemispheres, which means the hypergraph $(X\backslash Y, \{F \in \mathcal{F}:F\cap Y=\varnothing\})$ has proper 2 coloring.\par
    So we have $cd_2(\mathcal{F})\leq d$.
\end{proof}
The proof above is not Dol'nikov's original proof, but it still invokes the Borsuk-Ulam by its application of Lyusternik-Shnirel'man. Dol'nikov's original proof uses another version of Borsuk-Ulam theorem, as we shall see in the upcoming proposition that is used in the original proof.

\newtheorem*{prop}{Proposition}

\begin{prop}
    Let $C_1, C_2, ..., C_d$ be systems of nonempty compact convex sets in $\mathbb{R}$ (every $C_i$ is a family of compact convex sets) such that evey $C_i$ is intersecting, i.e. for any $A, B \in C_i$, $A \cap B \neq \varnothing$. Then there is a hyperplane intersecting all sets of the set system $\bigcup\limits_{i=1} ^d C_i$.
\end{prop}

\begin{proof}[\rm\bf{Proof of Proposition}]
    For a unit vector $v \in S^{d-1}$ starting at the origin, let $\ell_v$ denote the line passing the origin and oriented in the direction of $v$. For each $C_i$, consider the projections of the sets of $C_i$ onto $\ell_v$ and denote the intersection of these projections $I_i(v)=\bigcap\limits_{C \in C_i}\text{proj}_{\ell_v}(C)$. Such intersection must not be empty because we can apply Helly's theorem given that the projected intervals are 1 dimensional and any 2 intervals intersect. Furthermore, denote the midpoint of $I_i(v)$ by $m_i(v)$, which is a vector parallel to $\ell_v$ and $v$.
    \includegraphics[scale=1]{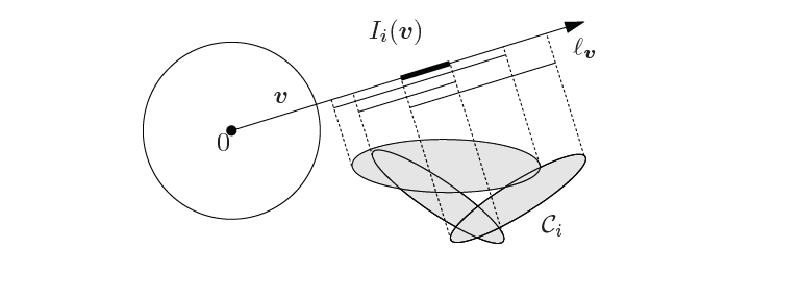}
    \begin{center}
        Credit: Matoušek
    \end{center}
    \par
    The goal is to find some $v$ such that all the midpoints coincide. We will prove that such a vector exists by defining some functions and applying the Borsuk-Ulam theorem.\par
    Define a continuous antipodal map $f:S^{d-1}\rightarrow \mathbb{R}^d$ by setting each component $f_i(v)=\langle m_i(v),v \rangle$. We see that $f$ is antipodal by noticing that $\ell_v=\ell_{(-v)}$. Then we have $\text{proj}_{\ell_v}(C)=\text{proj}_{\ell_{(-v)}}(C)$ for any $C \in C_i$ and $I_i(-v)=I_i(v)$. Consequently, $m_i(v)=m_i(-v)$ and we have $\langle m_i(v),v \rangle=-\langle m_i(v),-v \rangle=-\langle m_i(-v),-v \rangle$.\par
    We then define another continuous antipodal map $g: S^{d-1} \rightarrow \mathbb{R}^{d-1}$ by setting $g_i=f_i-f_d$ for $i=1,2,...,d-1$. The map is antipodal because $g_i(v)=f_i(v)-f_d(v)=-f_i(-v)+f_d(-v)=-(f_i(-v)-f_d(-v))=-g_i(-v)$. By Borsuk-Ulam, there exists $v \in S^{d-1}$ such that $g(v)=0 \Rightarrow f_i(v)-f_d(v)=0 $ for $ i=1,2,...,d$.\par
    By construction of $f$, we now have $\langle m_i(v),v \rangle=\langle m_d(v),v \rangle$ for all $i=1,2,...,d$ for some $v$. Then $0=\langle m_i(v),v \rangle - \langle m_d(v),v \rangle=\langle m_i(v)-m_d(v),v \rangle$, where $m_i(v)-m_d(v)$ is parallel to $v$. Hence it must be the case that $m_i(v)-m_d(v)=0$. Therefore we have $m_i(v)=m_d(v)$ for all $i=1,2,...,d$.\par
    Since all the midpoints coincide, we can consider the hyperplane perpendicular to $\ell_v$ at that point, which will intersect all the sets in $\bigcup\limits_{i=1} ^d C_i$.
\end{proof}

Given the above proposition, we can finally proceed to the second proof of Dol'nikov's theorem.

\newcommand{\F}{\mathcal{F}}

\begin{proof}[\rm\bf{Another Proof of Dol'nikov's Theorem}]
    Let there be a proper $d$-coloring of $KG(\F)$. Notice that if we have $d=\chi(KG(\F))$ then we have a proper $d$-coloring of the Kneser Graph (hence it is permissible to assume $d=\chi(KG(\F))$). Then $\F$ can be partitioned into $d$ set systems by color, i.e. $\F_1, \F_2, ..., \F_d$ where $\F_i=\{F \in \F: c(F)=i\}$ where $c$ is the coloring map on $KG(\F)$. Then any 2 sets from $\F_i$ must intersect because they share a color and the coloring is proper. \\
    \includegraphics[width=6cm, height=4cm]{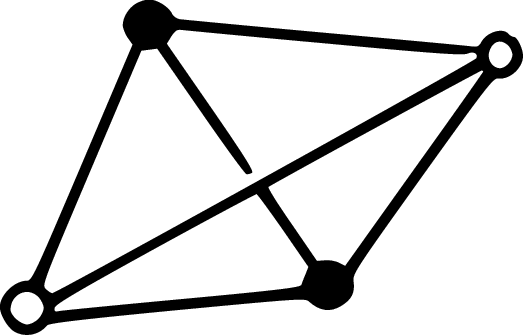}
    \begin{center}
        An illustration of the convex sets of the Kneser graph $4\choose2$ in $\mathds{R}^2$.
    \end{center} \par
    Next, we place the points of $X$ into $\mathbb{R}^d$ such that they are in general position, meaning no more than $d$ points can lie on the same hyperplane. Furthermore, define 
    \begin{equation*}
        C_i=\{\text{conv}(F): F \in \F_i \}
    \end{equation*}
    for $i=1,2,...,d$. Then $C_1, C_2, ..., C_d$ are systems of compact convex sets, and by the previous proposition we obtain a hyperplane $h$ intersecting all sets of $\bigcup\limits_{i=1}^d C_i=\{F \in \F\}$.\par
    
    We notice that since $h$ intersects every conv$(F)$, every $F$ must either completely consist of points from both of the open hemispheres divided by $h$, or it must contain a point lying on $h$ (otherwise, if $F$ consists of points from only one of the hemispheres and doesn't contain any point on $h$ then conv$(F)$ will not be intersected by $h$). Therefore if we color the points of $X$ on $h$ white, the points in one of the hemispheres red, and the points in the other hemisphere blue, then we have a proper 2 coloring of the hypergraph $(X\backslash Y, \{F \in \F: F\cap Y=\varnothing\})$ where $Y$ is the set of points of $X$ intersecting $h$ (i.e. the set of white points). There can be at most $d$ white points by the general position assumption. Hence $cd_2(\F)\leq d$.
 \end{proof}
 
\section{Conclusion}
As we can see, there exist some broad "recipe" as to how to apply the Borsuk-Ulam theorem to problems in discrete mathematics. Generally, one takes a discrete point set in general position, define some kind of clever set-covering on it, then use this cover to apply some version of Borsuk-Ulam or its derivatives (i.e, the ham sandwich theorem). The field of topological methods in combinatorics has proved rich in new ideas, and further research is needed to see how it can be applied to solve intractable problems in combinatorics and graph theory.


\section{References}
Matoušek, Jiří. Using the Borsuk–Ulam Theorem: Lectures on Topological Methods in Combinatorics and Geometry. Springer-Verlag, 2003.
\end{document}